\documentclass[12pt,twoside]{amsart}
\usepackage{mabliautoref}
\usepackage{amssymb,amsthm,amsmath}
\RequirePackage[dvipsnames,usenames]{xcolor}
\usepackage{hyperref}
\usepackage{mathtools}
\usepackage[abbrev,alphabetic]{amsrefs}
\usepackage[all]{xy}
\usepackage{tikz}
\usepackage{tikz-cd}
\usepackage{systeme}

\hypersetup{
	bookmarks,
	bookmarksdepth=3,
	bookmarksopen,
	bookmarksnumbered,
	pdfstartview=FitH,
	colorlinks,backref,hyperindex,
	linkcolor=Sepia,
	anchorcolor=BurntOrange,
	citecolor=MidnightBlue,
	citecolor=OliveGreen,
	filecolor=BlueViolet,
	menucolor=Yellow,
	urlcolor=OliveGreen
}

\makeatletter
\newcommand{\newreptheorem}[2]{\newtheorem*{rep@#1}{\rep@title}\newenvironment{rep#1}[1]{\def\rep@title{#2 \ref*{##1}}\begin{rep@#1}}{\end{rep@#1}}}
\makeatother
\makeatletter
\@namedef{subjclassname@2020}{%
	\textup{2020} Mathematics Subject Classification}
\makeatother

\DeclareMathOperator{\Spec}{Spec}
\DeclareMathOperator{\Proj}{Proj}

\newcommand{\bR}{\mathbb{R}}
\newcommand{\bQ}{\mathbb{Q}}
\newcommand{\bN}{\mathbb{N}}

\newcommand{\cO}{\mathcal{O}}

\newcommand{\cent}{\textup{centre}}
\newcommand{\codim}{\textup{codim}}

\newcommand{\Aut}{\textup{Aut}}
\newcommand{\Norm}{\textup{Norm}}

\newcommand{\Ht}[1]{H^{i}_{t}}

\newcommand{\stacks}[1]{\cite[\href{https://stacks.math.columbia.edu/tag/#1}{Tag #1}]{stacks-project}}

\newcommand{\ox}[1][X]{\mathcal{O}_{#1}}

\newtheorem{case}{Case}

\usepackage{xcolor}
\newcommand\myworries[1]{\textcolor{red}{#1}}

\title[Abundance for arithmetic threefolds]{Abundance theorem for threefolds in mixed characteristic}

\author{Fabio Bernasconi, Iacopo Brivio, and Liam Stigant}

\address{\'Ecole Polytechnique F\'ed\'erale de Lausanne, Chair of Algebraic Geometry
	(B\^atiment MA), Station 8, CH-1015 Lausanne} 
\email{fabio.bernasconi@epfl.ch}

\address{National Center for Theoretical Sciences, Taipei, 106, Taiwan}
\email{ibrivio@ncts.ntu.edu.tw}

\address{Department of Mathematics, Imperial College London, 180 Queen's Gate, 
	London SW7 2AZ, UK} 
\email{l.stigant18@imperial.ac.uk}


\begin{document}
	\begin{abstract}
		We prove the abundance theorem for arithmetic klt threefold pairs whose closed point have residue characteristic greater than 5.
		As a consequence, we give a sufficient condition for the asymptotic invariance of plurigenera for certain families of singular surface pairs to hold in mixed characteristic.
	\end{abstract}

\maketitle
\tableofcontents
	
\section{Introduction}
	
The Minimal Model Program (MMP for short) is a conjecture on the birational structure of algebraic varieties which aims to extend the theory of minimal models of surfaces to higher-dimensional varieties.
While the MMP is usually studied for projective varieties over an algebraically closed field, it is worth developing the program for projective schemes over more general bases, especially if one is interested in arithmetic applications. 
Although the MMP for surfaces over excellent rings can be considered a classical result (see \cite{Sha66} for the regular case and \cite{Tan18} for the logarithmic and singular case), only recent breakthroughs in commutative algebra in mixed characteristic (see \cite{And18, MS21, bhatt2020cohenmacaulayness}) have allowed progress on the birational geometry of arithmetic threefolds in the works \cite{takamatsu2021minimal,bhatt2020, XX22}, where the authors show that one can run the MMP for threefolds over excellent rings of mixed characteristic $(0, p>3)$. 
	
After showing the existence of minimal models, it is natural to ask whether the abundance conjecture, which states that a minimal model has semiample canonical divisor, holds for threefolds in mixed characteristic. 
The case of surfaces defined over a field was proven in increasing generality in \cite{fujino2012log, Tan14, tanaka2020abundance} while the case of threefolds over a perfect field characteristic $p>5$ is still open, though the non-vanishing conjecture being settled in \cite{XZ19, Wit} and various cases have been verified (\cite{DW19, Zha20}).
	
The aim of this article is to show the validity of the abundance conjecture for threefolds in mixed characteristic.
	
\begin{theorem}[\autoref{abundance}]\label{main-thm}
	Let $R$ be an excellent domain of finite Krull dimension, equipped with a dualising complex and whose residue fields of closed points have characteristic $p>5$.
	Let $\pi \colon X \to T$ be a projective morphism of quasi-projective $R$-schemes such that $\pi(X)$ is positive-dimensional.
	Suppose $(X,B)$ is a three-dimensional klt pair with $\mathbb{R}$-boundary. If $K_X+B$ is $\pi$-nef, then it is $\pi$-semiample.
\end{theorem}
	
The key idea of the proof is to apply abundance on the generic fibre $X_\eta$ where $\eta$ is the generic point of $\pi(X)$ and on the characteristic 0 part $X_{\mathbb{Q}} \to T_{\mathbb{Q}}$. 
From there we spread out the semiampleness over the positive characteristic points of $\pi(X)$ by a careful application of the MMP developed in \cite{bhatt2020} and Keel-Witaszek's base-point-free theorem \cite{witaszek2020keels}. 
One of the main tools of our proof is the following semiampleness criterion for EWM divisors.

\begin{proposition}[\autoref{EDsemiampleness2}]
	Let $X \to T$ be a projective contraction of normal quasi-projective schemes over $R$, where $\dim(X) \leq 3$.
	Let $L$ be an EWM $\mathbb{Q}$-Cartier $\mathbb{Q}$-divisor on $X$ such that its associated EWM contraction $f \colon X \to Z$ is equi-dimensional.
	If $L|_{X_{K(T)}}$ and $L|_{X_{\mathbb{Q}}}$ are semiample, then $L$ is semiample.
	\end{proposition}

A well-known and immediate consequence of the existence of minimal models and the abundance theorem is the finite generation of the canonical ring.
	
\begin{corollary}
    With $R$ as in \autoref{main-thm}, let $\pi \colon X\to T$ be a projective morphism of quasi-projective $R$-schemes such that $\pi(X)$ is positive-dimensional.
	Suppose $(X,B)$ is a three-dimensional klt pair with $\mathbb{Q}$-boundary. 
	Then the canonical $\ox[T]$-algebra
	\[R(K_X+B, \pi):=\bigoplus_{m \in \mathbb{N}} \pi_{*}\ox(\lfloor m(K_{X}+B)\rfloor)\]
	is finitely generated.
\end{corollary}

In characteristic $0$, finite generation of the canonical ring follows from the case for varieties of log general type (\cite{BCHM10}) and a result of Fujino and Mori (\cite[Theorem 5.2]{FM00}). However, their result requires a canonical bundle formula which is not available in the positive or mixed characteristic settings.

\begin{remark}
	An important step towards abundance consists of proving the non-vanishing conjecture, which is already a difficult result for threefolds over fields of characteristic $0$ (resp.~ over perfect fields of characteristic $p>5$), proven in \cite{Miy88} (see also \cite[Section 9]{FA}) (resp.~ \cite{XZ19, Wit}).
	However in our situation the non-vanishing follows immediately from non-vanishing on the generic fibre of $\pi$, which has dimension at most 2 and so the result is well-known.
\end{remark}
	
As a further application of \autoref{main-thm} we study the invariance of plurigenera for families of klt surface pairs in mixed characteristic.
It is well-known that invariance of plurigenera might fail over DVR of positive or mixed characteristic as shown in \cite{Lang, KU, Suh08, Bri20, Kol22}.
However it was proven in \cite{EH} that an asymptotic version of invariance of plurigenera holds for families of log smooth surface pairs if the ample model is not an elliptic fibration.
Using techniques of \cite{HMX18}, we use the MMP and \autoref{main-thm} to show an asymptotic invariance of plurigenera for families of klt surface pairs (possibly defined over imperfect fields), extending the work of Egbert and Hacon.
	
\begin{theorem}[\autoref{thm:ADIOP_final2}]\label{thm:ADIOP_final}
	Let $R$ be an excellent DVR such that its residue field $k$ has characteristic $p>5$.
	Let $X\to \Spec(R)$ be a projective contraction of quasi-projective $R$-schemes and suppose $(X,B)$ is a klt threefold pair.  Suppose that the following conditions are satisfied:
	\begin{enumerate}
		\item[(1)] $(X,X_{k}+B)$ is plt;
		\item[(2)] if $V$ is a non-canonical centre of $(X,X_k+B)$ contained in ${\mathbf{B}_{-}(K_{X}+B)}$, then $\dim (V_{k})=\dim (V) -1$.
	\end{enumerate}
		
	Then $X_k$ is normal. 
	Suppose further that at least one of the following holds:
	\begin{enumerate}
		\item $\kappa(K_{X_{k}}+B_{k}) \neq 1$; or
		\item $B_{k}$ is big over $\textup{Proj}(K_{X_{k}}+B_{k})$.
		\end{enumerate}	
	Then there is $m_{0} \in \mathbb{N}$ such that the equality
	$$h^{0}(X_{K},m(K_{X_{K}}+B_{K}))=h^{0}(X_{k},m(K_{X_{k}}+B_{k}))$$
	holds for all $m \in m_{0}\mathbb{N}$.
\end{theorem}

\textbf{Acknowledgments.}
The authors thank F.~Bongiorno, P.~Cascini, J.A.~Chen, E.~Floris, S.~Filipazzi and C.D.~Hacon for useful discussions on the content of this article. 
We are grateful to the referee for comments and suggestions, which considerably improved the presentation of the article.	
\begin{itemize}
	\item FB was partially supported by the NSF under grant number DMS-1801851, by a grant from the Simons Foundation; Award Number: 256202 and by the grant $\#$200021/169639 from the Swiss National Science Foundation;
	\item IB was supported by the National Center for Theoretical Sciences and a grant from the Ministry of Science and Technology, grant number MOST-110-2123-M-002-005;
	\item  LS is grateful to the EPSRC for his funding.
\end{itemize}

\section{Preliminaries}
	
\subsection{Notation}\label{notation}
	
\begin{enumerate}
		\item In this article, a base ring $R$ will always denote an excellent domain of finite Krull dimension admitting a dualizing complex $\omega_R^{\bullet}$. We assume $\omega_R^{\bullet}$ is normalised as explained in \cite{bhatt2020}.
		\item Given a scheme $X$, we write $X_\mathbb{Q}:= X
		\times_{\Spec(\mathbb{Z})} \Spec(\mathbb{Q})$.
		\item For an integral scheme $X$, we define the function field $K(X)$ of $X$ to be the localisation $\mathcal{O}_{X,\eta}$ at the generic point $\eta$ of $X$.
		\item We say $(X, \Delta)$ is a \emph{pair} if $X$ is an integral excellent normal scheme with a dualising complex, $\Delta$ is a $\mathbb{K}$-divisor and $K_X+\Delta$ is $\mathbb{K}$-Cartier, where $\mathbb{K}=\mathbb{Q}$ or $\mathbb{R}$.  When not specified we default to $\mathbb{K}=\mathbb{R}$. By $\dim(X)$ we mean the dimension of the scheme $X$. We say $X$ is a curve (resp.~ a surface, a threefold) if $\dim(X)=1$ (resp.~ 2, 3).
		\item If $(X, \Delta)$ is a pair and $\pi \colon Y \to X$ is a projective birational morphism of normal schemes, then we write
		$$K_Y+\pi_*^{-1}\Delta=\pi^*(K_X+\Delta)+\sum_i a(E_i, X, \Delta) E_i,$$
		where $E_i$ are the $\pi$-exceptional prime divisors and $a(E_i, X, \Delta)$ are called the \emph{discrepancies} of $E_i$ (with respect to $(X,\Delta)$).
		\item An open immersion $U \hookrightarrow X$ into a Noetherian scheme is said to be \emph{big} if $\codim_X(X \setminus U) \geq 2$.
		\item For the definitions of the singularities of the MMP in terms of discrepancies (as klt, plt, lc), we refer to \cite{kk-singbook,bhatt2020}.
		\item A \textit{contraction} is a proper morphism $f\colon X\to Y$ between integral schemes (resp.~  algebraic spaces) such that $f_*\cO_X=\cO_Y$.
		\item We say that a rational map of normal integral schemes $\varphi \colon X \dashrightarrow Y$ is a \emph{birational contraction} if it is a proper birational map such that $\varphi^{-1}$ does not contract any divisor. 
		\item Let $f\colon Y \to X$ be a contraction of Noetherian algebraic spaces. 
		We say an integral subspace $V \subseteq Y$ is \emph{contracted} by $f$ if $\dim f(V) < \dim V$. 
		\item Let $L$ be a nef line bundle on a scheme $X$ projective over $S$. We define the \emph{exceptional locus} $\mathbb{E}(L)$ of $L$ to be the union of integral subschemes $V$ of $X$ for which $L|_V$ is not big.
		\item If $X \to S$ is a projective morphism and $D$ is a $\mathbb{Q}$-Cartier $\bQ$-divisor on $X$, the \emph{diminished base locus of $D$ over $S$}  is  $$\mathbf{B}_{-}(D/S) = \bigcup_{A \, \mathbb{Q}\text{-divisor} \text{ ample}/S} \mathbf{SB}(D+A/S),$$ where $\mathbf{SB}(D+A/S)$ is the stable base locus of $D+A$ over $S$. If $S$ is clear from the context, we will simply write $\mathbf{B}_{-}(D)$.
		\item When $R$ is a discrete valuation ring (in short, DVR) with residue field $k$ and fraction field $K$, and $(X,\Delta)$ is a pair defined over $R$, we will denote the special, (resp.~ generic), fibre as $X_{k}$, resp.~ $X_K$. An analogous notation will be used for sheaves and $\mathbb{K}$-divisors on $X$.
\end{enumerate}
	
\subsection{Results from the three-dimensional MMP}
	
In this section we collect results on the three-dimensional MMP proven in \cite{bhatt2020}.
Recall that the existence of projective log resolutions for excellent schemes of dimension at most 3 has been established in \cite{CP19, CJS20} and we will use this fact repeatedly in the article.

We recall some terminology of the MMP we will use.
Let $\pi \colon X \to T$ be a projective morphism of quasi-projective $R$-schems and let $(X,\Delta)$ be a klt pair.
Suppose $\varphi \colon X \rightarrow Z$ is a contraction of a $(K_X+\Delta)$-negative  extremal ray. We say that $\varphi$ is \emph{of fibre type} if $\dim(Z) < \dim(X)$ and in this case we say that $X \to Z$ is a $(K_X+\Delta)$-Mori fibre space. 
If $\varphi$ is birational, we say it is a \emph{divisorial contraction} (resp.~  a \emph{flipping contraction}) if $\text{Exc}(\varphi)$ is purely of codimension 1 (resp.~  has codimension at least 2).
We say that $(X,\Delta)$ is a \emph{minimal model over $T$} if $K_X+\Delta$ is $\pi$-nef. We refer to \cite[Subsection 3.6]{BCHM10} for the definitions of weak log canonical models and $D$-non-positive birational contractions.

\begin{notation}
	In the following, the characteristics of the residue fields of $R$ are different from $2,3$ and $5$. 
\end{notation}

\begin{theorem}[{Base-point-free theorem, \cite[Theorem G]{bhatt2020}}]\label{t-bpf}
	Let $(X,\Delta)$ be a $\mathbb{Q}$-factorial klt threefold pair with $\mathbb{R}$-boundary and let $f \colon X \to T$ be a projective morphism of quasi-projective $R$-schemes.
	Let $L$ be an $f$-nef Cartier divisor such that $L-(K_{X}+\Delta)$ is $f$-big and $f$-nef.
	If the image of $X$ in $T$ is positive-dimensional, then $L$ is $f$-semiample. 
\end{theorem}

\begin{theorem}[{Running the MMP, \cite[Theorem F]{bhatt2020}}]  \label{MMP} 
	Let $(X,\Delta)$ be a $\mathbb{Q}$-factorial dlt threefold pair with $\mathbb{R}$-boundary and let $f \colon X \to T$ be a projective morphism of quasi-projective $R$-schemes with positive-dimensional image. 
	Then we can run a $(K_{X}+\Delta)$-MMP over $T$. 
	If $K_{X}+\Delta$ is pseudo-effective over $T$, then any sequence of steps of a $(K_X+\Delta)$-MMP terminates. 
	Further, a $(K_X+\Delta)$-MMP over $T$ with scaling by an ample divisor $A$ will terminate.
\end{theorem}

	
	
\begin{corollary}\label{Q-factorial}
	Let $(X,B)$ be a plt threefold pair with $\mathbb{R}$-boundary quasi-projective over $R$. 
	Then there exists a small projective birational morphism $f \colon Y \to X$ such that $Y$ is $\mathbb{Q}$-factorial.
\end{corollary}
	
\begin{proof}
    Let $\varphi \colon W \to X$ be a log resolution of $(X,B)$.
    By \autoref{MMP}, we can run a $(K_W+\varphi_*^{-1}B+\text{Exc}(\varphi))$-MMP over $X$ which ends with a birational contraction $f\colon Y \to X$ such that  $(Y,f_{*}^{-1}B+\text{Exc}(f))$ is a $\mathbb{Q}$-factorial dlt pair with $K_{Y}+f_{*}^{-1}B+\text{Exc}(f)$ nef over $X$. 
    Then by the negativity lemma (\cite[Lemma 2.14]{bhatt2020}) we have $K_{Y}+f_{*}^{-1}B+\text{Exc}(f)+F=f^{*}(K_{X}+B)$, where $F$ is effective and $f$-exceptional. Since $(X,B)$ is plt we have that $\text{Exc}(f)+F=0$, so $f$ is small.
\end{proof}

\subsection{Algebraic spaces}
	
    We refer to \stacks{0ELT} for the definition of algebraic spaces and their general theory. We record here a few key results to be used later. First, Stein factorisation exists for algebraic spaces. We fix an excellent Noetherian base scheme $S$.
	
	\begin{theorem}[Stein factorisation, \stacks{0A1B}]\label{Stein}
	  Let $f \colon X \to Y$ be a proper morphism of Noetherian algebraic spaces over $S$. 
	  Then there is a morphism $f' \colon X \to Y'$, together with a finite morphism $\pi \colon Y' \to Y$, factorising $f$ into $f=\pi \circ f'$  such that
	  \begin{itemize}
	      \item $f'$ is proper and surjective;
	      \item $f'_{*}\ox = \ox[Y']$;
	      \item $Y'=\underline{\textup{Spec}}_{Y}(f_{*}\ox)$;
	      \item and $Y'$ is the normalisation of $Y$ in $X$.
	  \end{itemize}
	  We call $f'$ the Stein factorisation of $f$.
	 \end{theorem}
	 
	 In particular if $X$ is normal in \autoref{Stein}, then so is $Y'$. Moreover if $X,Y$ are schemes then $Y'$ is a scheme and $f'$ agrees with the usual notion of Stein factorisation.
	 We also have the following descent result for projective contractions of algebraic spaces.
	
\begin{lemma}\label{as-ext}
Let $f \colon W \to X$ and $g \colon W \to Y$ be projective contractions of Noetherian integral normal algebraic spaces over $S$. 
Suppose that every proper curve $C \subset W$ contracted by $f$ is contracted by $g$. 
Then there is a unique contraction $h\colon X \to Y$ with $g=h \circ f$.
\end{lemma}

\begin{proof}
First, note that any $h\colon X\to Y$ such that $g=h \circ f$ is necessarily a contraction.
Consider $(g\times_S f)\colon W \to X \times_S Y$ and let $\phi \colon W \to \Gamma$ be the contraction part of its Stein factorisation.
Thus $\Gamma$ is an integral, normal algebraic space which is projective over $X$. If $\gamma\colon \Gamma\to X$ is the induced morphism, it is then enough to show that $\gamma$ is an isomorphism.
	    
Let $x\in X$ be any point, and let $F:=\gamma^{-1}(x)$. Then $\phi^{-1}(F)=f^{-1}(x)$ is contracted by $g$, hence by $\phi$, so $\gamma$ is quasi-finite.
	    
Let $\xi\in X$ be the generic point. As $f$ is a contraction, we have $H^0(W_\xi,\cO_{W_\xi})=\kappa(\xi)$. As $\phi$ is a contraction and Stein factorisation commutes with flat base-change, we have that $\phi_\xi\colon W_\xi\to\Gamma_\xi$ is a contraction as well, thus $H^0(\Gamma_\xi,\cO_{\Gamma_\xi})=H^0(\Gamma_\xi,\phi_{\xi,*}\cO_{W_\xi})=\kappa(\xi)$. By \cite[\href{https://stacks.math.columbia.edu/tag/0AYI}{Tag 0AYI}]{stacks-project} we then have that $\gamma$ is a contraction, and by \cite[\href{https://stacks.math.columbia.edu/tag/082I}{Tag 082I}]{stacks-project} we conclude it is an isomorphism.
\end{proof}
	
	
	It will prove useful to know that proper algebraic spaces are schemes on a big open set.
	
	\begin{lemma}\label{l-cod1}
	    Let $X$ be a proper algebraic space over $S$. 
	    Then there is a big open immersion of a scheme $U\to X$. 
	    If $X$ is normal, we can choose $U$ to be regular.
	\end{lemma}
	
	\begin{proof}
	    By \stacks{0ADD}, for each codimension $1$ point $P \in X$ there is an open subspace $U_{P}$ containing $P$ which is a scheme.
	    Take the open subspace   $U=\bigcup_{\codim_X(P)=1} U_{P},$
	    of $X$.
	    By \cite[\href{https://stacks.math.columbia.edu/tag/01JJ}{Tag 01JJ}]{stacks-project} we observe that in fact $U$ is a scheme. Note that $U$ is a sheaf on the Zariski topology since, by definition, it is a sheaf on the finer fppf topology, \stacks{025Y}. 
	    
	    If $X$ is normal, then so too are the $U_{P}$. In particular, after shrinking them as needed we may suppose that each $U_{P}$ is regular and thus that $U$ is regular.
\end{proof}

\subsection{Semiample and EWM line bundles}
	
In this subsection we recall some basic results about semiample and EWM line bundles we will need later on. 
We fix $S$ to be an excellent Noetherian separated base scheme.

\begin{definition}
	Let $\varphi \colon X \to S$ be a proper morphism. A line bundle $L$ on $X$ is said to be \textit{semiample over $S$} if there exists $m>0$ such that $L^{m}$ is globally generated over $S$, \emph{i.e.} the natural morphism $\varphi^*(\varphi_*L^{m}) \to L^m$ is surjective.
\end{definition} 
	
\begin{theorem}\label{t-semiamplecontraction}
Let $X$ be a normal projective scheme over $S$ and let $L$ be a line bundle on $X$. Then the following are equivalent.
\begin{enumerate}
	\item $L$ is semiample over $S$;
	\item there is a contraction $f\colon X\to Z$ over $S$ such that $f$ is induced by the linear system $|L^m/S|$ for all sufficiently divisible $m$;
	\item there is a contraction $f\colon X\to Z$ over $S$ such that $L\sim_{\mathbb{Q}} f^{*}A$ for $A$ ample $\mathbb{Q}$-Cartier $\mathbb{Q}$-divisor on $Z$.
\end{enumerate}
Moreover if there is such a $Z$ then it is normal.
\end{theorem}
	
\begin{proof}
	The direction (a) $\Rightarrow$ (b) $\Rightarrow$ (c) is the content of \cite[Theorem 2.1.26]{Laz04}. The implication (c) $\Rightarrow$ (a) follows straight from the definition of ampleness.
	 As $f$ is a contraction, then $Z$ is normal.
\end{proof}
	
The morphism $f$ is the same in both (b) and (c) of \autoref{t-semiamplecontraction}, and it is called the \textit{semiample contraction of $L$}. 

\begin{definition}
	Let $\varphi \colon X \to S$ be a proper morphism of schemes. 
	A nef line bundle $L$ on $X$ is said to be \emph{EWM over $S$} if there exists a proper $S$-morphism $f \colon X \rightarrow Y$ to an algebraic space $Y$ proper over $S$ such that an integral closed subscheme $V \subset X$ is contracted if and only if $L|_V$ is not big. 
\end{definition}
	
By \autoref{Stein}, we can suppose $f$ is a contraction and we call this the \emph{EWM contraction} associated to $L$, which is unique up to isomorphism by \autoref{as-ext}.

The definition of semiample (resp.~ EWM) extends naturally to $\mathbb{Q}$-Cartier $\bQ$-divisors (resp.~ $\mathbb{R}$-Cartier $\bR$-divisors) on proper normal schemes over $S$.
We say that an $\mathbb{R}$-Cartier $\bR$-divisor $D$ is semiample if there exist $r_i>0$ and $L_i$ semiample Cartier divisors such that $D \sim_{\mathbb{R}}\sum_i r_{i}L_{i}$. A natural extension of condition (c) in \autoref{t-semiamplecontraction} is that $D$ is semiample if and only if there is a morphism $f \colon X \to Z$ of $S$-schemes such that $D\sim_{\mathbb{R}} f^*A,$ where $A$ is an ample $\mathbb{R}$-divisor over $S$.

\subsubsection{Semiampleness criteria}

We recall Keel-Witaszek's base-point-free theorem, which will play a crucial role in the proof of abundance.
	
\begin{theorem}[\cite{Keel, witaszek2020keels}]
	Let $L$ be a nef line bundle on a scheme $X$ projective over $S$. Then $L$ is semiample over $S$ if and only if both $L|_{\mathbb{E}(L)}$ and $L|_{X_{\mathbb{Q}}}$ are semiample.
\end{theorem}
	
We will need the following descent result for semiampleness on normal schemes (cf. \cite[Lemma 2.1]{FL19}).
	
\begin{lemma}\label{pullback}
	Let $f \colon X \to Y$ be a proper surjective morphism of integral, excellent Noetherian schemes over $S$. Suppose that $Y$ is normal and $L$  is a line bundle on $Y$ such that $f^{*}L$ is semiample over $S$.
	Then $L$ is semiample over $S$.
\end{lemma}
	
\begin{proof}
The proof is similar to \cite[Lemma 2.10]{Keel}.	
We may freely assume that $X$ is normal. 
Let $X \xrightarrow{\varphi} Z \xrightarrow{\psi} Y$ be the Stein factorisation of $f$, where $\varphi$ is a contraction and $\psi$ is a finite map. 
We first show that $\psi^*L$ is semiample. Take $m>0$ such that $f^*L^{ m}$ is base point free. By the projection formula $H^0(X, f^*L^{m})=H^0(Z, \psi^*L^{m})$ and so $\psi^*L^{m}$ is base-point-free.
		
We can thus assume that $f$ is a finite morphism of degree $d$.
By \stacks{0BD3}, there exists a norm function $\Norm_f \colon f_*\mathcal{O}_X \to \mathcal{O}_Y$ of degree $d$ for $f$ which induces a group homomorphim $\Norm_f \colon \text{Pic}(X) \to \text{Pic}(Y)$ by \stacks{0BCY}. Take $m>0$ such that $f^*L^{m}$ induces the semiample contraction, and let $y\in Y$ be a point. Then there is a section $s\colon \cO_X\to f^*L^{m}$ not vanishing at any of the points in $f^{-1}(y)$. By \stacks{0BCY} and \stacks{0BCZ} we then construct a section $\Norm_f(s)\colon \cO_{Y}\to L^{md}$ not vanishing at $y$, concluding.
\end{proof}

	We will need a similar, but slightly weaker result for algebraic spaces. First we make the following observation.

	\begin{lemma}\label{com-big}
	    Let $f \colon Y \to X$ be a contraction of integral normal proper $S$-schemes. Let $L$ be a line bundle on $X$ nef over $S$. Let $V \subset X$ (resp. $V' \subset Y$) be an integral closed subscheme.
	    Suppose $f(V')=V$. Then $f^{*}L|_{V'}$ is big over $S$ if and only if $L|_V$ is big over $S$ and $\dim (V)=\dim (V')$.  
	\end{lemma}
	
	\begin{proof}
	    
	   Let $d$ be the dimension of $V'$. Since $f^{*}L$ is nef, it is big on $V'$ if and only if $(f^{*}L)^{d} \cdot V'>0$.
	    Hence by the projection formula (\cite[Proposition VI.2.11]{k-rat-curves}) it is big on $V'$ if and only if $L^{d} \cdot V> 0$. In turn this occurs if and only if $\dim (V) =d$ and $L$ is big on $V$.
	  \end{proof}
	
	\begin{lemma}\label{pp-EWM}
	Let $f \colon Y\to X$ be a contraction of integral normal projective $S$-schemes.
	A line bundle $L$ on $X$ is EWM if and only if $f^{*}L$ is EWM.	
	\end{lemma}
	
\begin{proof}
  Suppose first that $L$ is EWM and let $g \colon X \to Z$ be the associated EWM contraction. 	
  We claim that $h=g \circ f$ contracts an integral subscheme $V$ of $Y$ if and only if $f^{*}L|_{V}$ is not big. 	
  By \autoref{com-big}, $f^{*}L|_{V}$ is not big if and only $\dim (f(V))< \dim (V)$ or $L|_{f(V)}$ is not big, showing that $h$ is the EWM contraction associated to $f^*L$. 
	  
  Now suppose that $f^{*}L$ is EWM. Let $g \colon Y \to Z$ be the associated EWM contraction. 
  By \autoref{as-ext} there exists a contraction $h \colon X \to Z$ with $g=h\circ f$. We are only left to show that $h$ is the EWM contraction associated to $L$.
  
  Let $V\subset X$ be an integral subscheme of dimension $d$. We can choose an integral $V'$ lying over $V$ of dimension $d$ by cutting $f^{-1}(V)$ with hyperplanes and taking a dominant component. 
  By \autoref{com-big} we see that
  $L|_V$ is not big if and only if $f^*L|_{V'}$ is not big. By hypothesis $f^*L|_{V'}$ is not big if and only if $V'$ is contracted. This is equivalent to $V$ being contracted, which concludes the proof.
\end{proof}

    \begin{remark}
        Clearly, if $L$ is an EWM line bundle on $X$ and $T$ is any integral closed subscheme, then $L|_{T}$ is EWM.    
    \end{remark}

	
	
		

	\subsubsection{Semiample line bundles over DVRs}
	We study how the spaces of global sections of $L$ behave in a family of normal projective varieties over a DVR $R$.
	
	Given a $\mathbb{Q}$-Cartier $\bQ$-divisor $L$ on a normal variety $X$ over a field $k$, we denote by $\kappa(L)$ its Iitaka dimension (see \cite[Definition 2.1.3]{Laz04}).
	\begin{lemma}\label{lemma_DIOK}
		Let $R$ be a DVR and let $\pi \colon X \to \Spec(R)$ be a normal integral projective $R$-scheme such that $X_k$ is normal.
		If $L$ is a  semiample $\bQ$-Cartier $\bQ$-divisor on $X$, then $\kappa(L_k)=\kappa(L_K)$.
	\end{lemma}
	
	\begin{proof}
		Let $f\colon X \to Z$ be the semiample contraction of $L$ over $R$, let $\delta\colon Z\to\Spec(R)$ be the structure morphism.
		As $Z$ is integral and $\delta$ is dominant, then it is flat by \cite[Proposition 9.7]{Ha77} and thus equi-dimensional. Let $d$ be the dimension of the fibers of $\delta$, and let $A$ be an ample $\bQ$-divisor on $Z$ such that $L\sim_{\bQ}f^\ast A$. By the projection formula (\cite[\href{https://stacks.math.columbia.edu/tag/01E8}{Tag 01E8}]{stacks-project}) and asymptotic Riemann-Roch (\cite[Theorem VI.2.15]{k-rat-curves}), for each $t \in \Spec(R)$ we have
		\begin{equation*}
			\begin{split}
				h^0(X_t,mL_t)&=h^0(Z_t,f_{t,*} \cO_{X_t}\otimes\cO_{Z_t}(mA_t))\\
				&=\mathrm{rk} (f_{t,*} \cO_{X_t})\frac{(mA_t)^d}{d!}+O(m^{d-1})
			\end{split}
		\end{equation*}
		for all $m> 0$ sufficiently divisible. Thus we conclude $\kappa(L_t)=d$ for each $t \in \Spec(R)$.
	\end{proof}

	\begin{lemma}\label{l-stein-invariance}
		Let $R$ be a DVR and let $\pi \colon X \to \Spec(R)$ be a normal integral projective $R$-scheme such that $X_k$ is normal. 
		Let $L$ be a $\bQ$-Cartier $\bQ$-divisor on $X$, semiample over $R$ and let $f \colon X \to Z$ be the semiample contraction induced by $L$.
		Then the following are equivalent:
		\begin{enumerate}
			\item[(1)] $f_{k,*} \cO_{X_k} = \cO_{Z_k}$;
			\item[(2)] $h^0(X_k, mL_k)=h^0(X_K, mL_K)$ for all $m\geq 0$ sufficiently divisible.
		\end{enumerate}
	\end{lemma}
	
	\begin{proof}
		Let $A$ be an ample $\bQ$-divisor on $Z$ such that $L \sim_{\bQ}f^*A$. 
		By the projection formula we have 
		\begin{equation}\label{e-globalsectionsPF}
			h^0(X_t,mL_t)=h^0(Z_t,f_{t,\ast}\cO_{X_t}\otimes\cO_{Z_t}(mA_t))
		\end{equation}
		for all sufficiently divisible $m$ and all $t\in \Spec (R)$. By flat base change we have $f_{K,\ast}\cO_{X_K}=\cO_{Z_K}$.
		
		$(1) \Rightarrow (2)$. Suppose that $f_{k,\ast}\cO_{X_k}=\cO_{Z_k}$. Then the right hand side of Equation (\ref{e-globalsectionsPF}) coincides with $\chi(Z_t,mA_t)$ when $m\gg 0$ by Serre vanishing. Hence we conclude by the invariance of the Euler characteristic in a flat family (\cite[Theorem 9.9]{Ha77}).
		
		$(2) \Rightarrow (1)$. By Grauert's theorem (\cite[Corollary III.12.9]{Ha77}) the natural restriction map
		$H^0(X,\mathcal{O}_X(mL))\to H^0(X_k,\mathcal{O}_{X_k}(mL_k))$
		is surjective for all $m \geq 0$ sufficiently divisible. Hence $f_k$ is the semiample contraction of $L_k$ by \autoref{t-semiamplecontraction}, in particular $f_{k,\ast}\cO_{X_k}=\cO_{Z_k}$. 
	\end{proof}
	
	\begin{remark}\label{r-connected fibers v contraction}
	Suppose that $Z_{k}$ is normal in \autoref{l-stein-invariance} and let $X_{k}\to Y_{k} \xrightarrow{g} Z_{k}$ be the Stein factorisation of $f_{k}$. If $k$ is a field of characteristic $0$ then $g$ is birational and finite, hence an isomorphism.
	On the other hand if $k$ is a positive characteristic field then $g$ may be a non-trivial purely inseparable morphism of normal varieties. This is an obstruction to lifting sections of $mL_k$ (see \cite{Bri20} for an explicit construction with $L=K_X+B$). For this reason, a crucial step in \autoref{thm:ADIOP_SA} will be showing $f_{k,*} \mathcal{O}_{X_k}=\mathcal{O}_{Z_k}$.
    \end{remark}
	
\subsection{Adjunction for non-$\mathbb{Q}$-factorial threefolds}
    In this subsection we fix $R$ to be an excellent DVR with residue field $k$ of characteristic $p>5$.
    We show that threefold plt centres appearing as fibres over $R$ are normal, extending some previous results in the literature to the non-$\mathbb{Q}$-factorial case. 
	
	We start by recalling normality of plt centres in the $\mathbb{Q}$-factorial case. 
	
	\begin{theorem}[{\cite[Corollary 7.17, Remark 7.18]{bhatt2020}}]\label{invAdj}
	 	Let $(X, S + B)$ be a three-dimensional plt pair where $S$ is reduced and $\lfloor{B\rfloor}=0$ and suppose that the closed points of $X$ have characteristic $p \neq 2, 3$ and $5$. 
	 	If $B$ has
		standard coefficients,
		then
		$S$ is normal.
	\end{theorem}
	
	We will be particularly interested in the case where $S=X_{k}$ is the central fibre over a DVR. 
	In this case we write $\Delta_{k}=\Delta|_{X_k}$. Note that, since $X_{k}$ is Cartier, this is well defined even if $\Delta$ is not $\mathbb{Q}$-Cartier.
	Moreover, if $(X,X_{k}+\Delta)$ is plt then the different $\textup{Diff}_{X_k}(\Delta)$ coincides with $\Delta_{k}$ by \cite[Proposition 4.5]{kk-singbook}.
	
\begin{corollary}\label{normality}
	Let $(X,X_{k}+\Delta)$ be a $\mathbb{Q}$-factorial threefold plt pair, quasi-projective over $R$. 
	Then $X_{k}$ is normal and $(X_k, \Delta_k)$ is klt.
\end{corollary}
\begin{proof}
	Since $X$ is $\mathbb{Q}$-factorial, the pair $(X, X_k)$ is plt.
	By \autoref{invAdj} we conclude $X_{k}$ is normal and thus by adjunction (\cite[Lemma 4.8]{kk-singbook}) we conclude $(X_k, \Delta_k)$ is klt.
\end{proof}
	
	Given suitable vanishing results over $k$, we can extend this adjunction result to more general situations by making use of lifting arguments.
	
\begin{proposition}\label{push-lift}
	Let $S$ be the spectrum of a local Artinian ring and $S \hookrightarrow S'$ be a closed immersion defined by a square-zero ideal $I$.  Let $f\colon Y \to S$, and $h\colon X \to S$ be flat morphisms and let $g\colon Y \to X$ be a morphism of $S$-schemes. 
	Suppose that $g_{*} \ox[Y]=\ox$, $R^{1}g_{*} \ox[Y] = 0$ and $Y$ has a flat lifting $f' \colon Y' \to S'$. Then there exists a flat lifting $h' \colon X'\to S'$ and a morphism $g' \colon Y' \to X'$ making the following commutative diagram:
		
	\[\begin{tikzcd}
		Y \arrow[r] \arrow[d, "g",swap]  \arrow[bend right=60,swap, "f"]{dd}
		& Y' \arrow[d, "g'"] \arrow[bend left=60, "f'"]{dd} \\
		X \arrow[d, "h",swap] \arrow[r] & X' \arrow[d, "h'"] \\
		S \arrow[r]                        & S'    .   
	\end{tikzcd}\]
	Moreover,  $g'_{*} \ox[Y']=\ox[X']$ and $R^{1}g'_{*} \ox[Y'] = 0$.
\end{proposition}

\begin{proof}
    This is the construction of \cite[Theorem 3.1]{cynk2009small}.
\end{proof}

\begin{theorem}\label{adj-push}
	Let $X$ be a normal projective $R$-scheme such that $X_{k}$ is normal. 
	Let $f \colon X \to Z$ be a contraction over $R$ and suppose that $$f_{k}\colon X_{k} \xrightarrow{g_{1}} Y_{1} \xrightarrow{h_{1}} Z_{k}$$ is the Stein factorisation of $f_{k}$. If $R^{1}g_{1,*} \ox[X_{k}]=0$, then $h_1$ is an isomorphism. In particular, $f_k$ is a contraction and $Z_k$ is normal.
\end{theorem}
	
\begin{proof}
	Since we are only interested in the special fibre, we can replace $R$ with its completion at its maximal ideal $\mathfrak{m}$ without any loss of generality.
	Write $R_{i}:=R/\mathfrak{m}^{i}$ then let $X_{i}:=X \times_{\Spec(R)} \Spec(R_{i})$, $Z_{i}:=Z\times_{\Spec(R)} \Spec(R_{i})$ and $f_{i}=f\times_{\Spec(R)} \Spec(R_{i})\colon X_{i} \to Z_{i}$.
	Then $f_{1}$ factors as $f_{1}\colon X_{1} \xrightarrow{g_{1}} Y_{1} \xrightarrow{h_{1}} Z_{1}$ where $R^{1}g_{1,*}\ox[X_{1}]=0$, so by \autoref{push-lift} we can lift $g_{1}\colon X_{1} \to Y_{1}$ to $g_{i}\colon X_{i} \to Y_{i}$ over $R_{i}$ such that the following diagram commutes.
	
	\[\begin{tikzcd}
		X_{1} \arrow[r] \arrow[d, "g_{1}"] & X_{2} \arrow[r] \arrow[d, "g_{2}"] & \dots \\
		Y_{1} \arrow[r] \arrow[d, "h_{1}"] & Y_{2} \arrow[d, dotted, "h_{2}"] \arrow[r]  & \dots \\
		Z_{1} \arrow[r]                    & Z_{2} \arrow[r]                    & \dots
	\end{tikzcd}\]
	For $i \in \mathbb{Z}_{>0}$, the morphisms $h_{i}$ are defined as follows. The underlying topological map is just $h_{1}$ and the map $\ox[Z_{i}] \to h_{i,*}\ox[Y_{i}]$ comes from the map ${\ox[Z_{i}] \to f_{i,*}\ox[X_{i}]}$ and the identification $f_{i,*}\ox[X_{i}]=h_{i,*}g_{i,*}\ox[X_{i}]\simeq h_{i,*}\ox[Y_{i}]$.
	Each $h_{i}$ is finite, and thus by
	\cite[\href{https://stacks.math.columbia.edu/tag/09ZT}{Tag 09ZT}]{stacks-project} we have that the compatible system $\left\{Y_{i} \to Z_i \right\}$ lifts to a finite morphism $Y \to Z$ over $R$. By \cite[\href{https://stacks.math.columbia.edu/tag/0A42}{Tag 0A42}]{stacks-project} there is a factorisation ${f\colon X \xrightarrow{g} Y \xrightarrow{h} Z}$, where $g_{*}\ox = \cO_Y$, because $g_{i,*}\cO_{X_i}=\cO_{Y_i}$ for all $i$. Similarly $h$ is a finite morphism. 
	
	Therefore $f \colon X \xrightarrow{g} Y \xrightarrow{h} Z$ is the Stein factorisation for $f$, but since $f$ is a contraction of normal schemes we conclude that $h$ has to be an isomorphism.
	In particular, $h_1$ is an isomorphism and $Z_{k}=Y_{k}$.
\end{proof}

	\begin{lemma}\label{invAdj2}
		Let $X \to \Spec(R)$ be a projective contraction of normal schemes. Suppose that
		\begin{enumerate}
			\item $(X, X_k+\Delta)$ is a threefold plt pair and $X_k$ is normal;
			\item there is a projective contraction $f \colon X \to Z$ over $R$ such that $-(K_{X_{k}}+\Delta_{k})$ is $f_{k}$-big and $f_k$-nef.
		\end{enumerate}  
	Then $Z_{k}$ is normal and $f_{k,*}\ox[X_{k}]=\ox[Z_{k}]$. Further, if $f$ is birational and $B:=f_{*}\Delta$, then $(Z, Z_k+B)$ is plt and $(Z_{k},B_{k})$ is klt.
	\end{lemma}
	
	\begin{proof}
		Since $X_{k}$ is normal, the pair $(X_{k},\Delta_{k})$ is klt by adjunction (\cite[Lemma 4.8]{kk-singbook}). 
		Let $$f_{k}\colon X_{k} \xrightarrow{\bar{f}_{k}} \bar{Z_k} \xrightarrow{h_k} Z_{k}$$ be the Stein factorisation of $f_k$. 
		Since $-(K_{X_{k}}+\Delta_{k})$ is $\bar{f}_{k}$-big and $\bar{f}_{k}$-nef, we conclude $R^{i}\bar{f}_{k,*}\ox[X_{k}]=0$ for $i> 0$ by \cite[Proposition 3.2]{Tan18} and \cite[Theorem 5.7]{BT22}.
		By \autoref{adj-push} $h_k$ is an isomorphism, $f_{k,*}\ox[X_{k}]=\ox[Z_{k}]$ and $Z_{k}$ is normal.
		
		Suppose now $f$ is birational. As $(X,X_{k}+\Delta)$ is plt, so is $(Z,Z_k+B)$ as the plt centre $Z_k$ is not contracted. Hence $(Z_k,B_k)$ is klt by adjunction.	
		\end{proof}

We are now able to prove the normality of the special fibre in a plt family not necessarily $\mathbb{Q}$-factorial.

\begin{corollary}\label{invAdj3}
	Let $X \to \Spec(R)$ be a projective contraction of normal schemes.
	Suppose $(X,X_{k}+\Delta)$ is a threefold plt pair. Then $X_{k}$ is normal and $(X_{k}, \Delta_{k})$ is klt.	
\end{corollary}
	
	\begin{proof}
		Let $f\colon Y \to X$ be a small $\mathbb{Q}$-factorialisation given by \autoref{Q-factorial} and write $K_Y+Y_k+\Delta_Y=f^*(K_X+X_k+\Delta)$. Then $(Y,Y_{k}+\Delta_{Y})$ is a $\mathbb{Q}$-factorial plt pair and hence $Y_{k}$ is normal by \autoref{normality}. By construction $f$ is a $(K_{Y}+Y_k+\Delta_{Y})$-trivial birational contraction, thus \autoref{invAdj2} ensures the result.
	\end{proof}
	
	\subsection{MMP in families}
	
	In this subsection we fix $R$ to be an excellent DVR with residue field $k$ of characteristic $p>5$.
	We collect some results on the MMP in families over $R$ that we will use in \autoref{s-inv-plurigenera}.
	
We start by recalling that discrepancies do not decrease while running an MMP.
		\begin{lemma}\label{l:increase-discr}
		Let 
		\[
		\xymatrix{
			X \ar[dr]_{f}   \ar@{-->}[rr]^{\varphi} &  &  X' \ar[dl]^{f'}  \\
			&Z & ,
		}
		\]
		be a commutative diagram,  where $(X,\Delta)$ and $(X', \Delta')$ are pairs and the morphisms $f$ and $f'$ are birational.
		Assume that
		\begin{enumerate}
			\item $f_*\Delta=f'_*\Delta'$;
			\item $-(K_X+\Delta)$ is $f$-nef;
			\item $K_{X'}+\Delta'$ is $f'$-nef. 
		\end{enumerate}
		Then for any exceptional divisor $E$ over $Z$ we have $$a(E, X, \Delta ) \leq a(E, X', \Delta').$$
		Furthermore, if $-(K_X+\Delta)$ is $f$-ample and $f$ is not an isomorphism above the generic point of $\cent_X(E)$, then
		$$ a(E, X, \Delta ) < a(E, X', \Delta').$$
	\end{lemma}
	\begin{proof}
		Let $E$ be an exceptional divisor over $Z$ and let $W$ be a normal scheme with projective birational morphisms $g \colon W \to X$ and $g' \colon W \to X'$ such that $\textup{centre}_W(E)$ is divisorial. Denote by $\pi:= f \circ g=f' \circ g'$. Let $E_{i}$ be the $\pi$-exceptional divisors. By hypothesis, the following divisor
		$$\Gamma:= \sum_i ( a(E_i, X, \Delta)-a(E_i, X', \Delta'))E_i $$ 
		is $\pi$-nef and it is a sum of $\pi$-exceptional divisors.
		Therefore, by the negativity lemma (\cite[Lemma 2.14]{bhatt2020}), we have $-\Gamma\geq 0$. Finally, if $\Gamma$ is not numerically $\pi$-trivial over $\cent_Z(E)$, we conclude the strict inequality.
	\end{proof}
	
	As we will need to impose some conditions on the base loci of log canonical divisors and non-canonical centres for our applications to invariance of plurigenera (see \autoref{ex-kawamata}), we study the behaviour of the diminished base locus $\mathbf{B}_{-}(K_X+\Delta)$ under the steps of the MMP.

	\begin{lemma}\label{l-stable-base-loci}
		Let $(X,\Delta)$ be a klt pair, projective over a quasi-projective $R$-scheme $T$.
		Let $f\colon X \dashrightarrow Y$ be a step of a  $(K_X+\Delta)$-MMP over $T$ and write $\Delta_Y=f_*\Delta$.
		Let
		\begin{equation*}
			\xymatrix{
				& W \ar[dr]^q \ar[dl]_p & \\
				X \ar@{-->}[rr]^{f} & & Y}
		\end{equation*} 
		be a resolution of indeterminacies of $f$.
		Then $q^{-1}\mathbf{B}_{-}(K_Y+\Delta_Y) \subset p^{-1}\mathbf{B}_{-}(K_X+\Delta).$
	\end{lemma}
	
	\begin{proof}
		By the negativity lemma, $p^*(K_X+\Delta)=q^*(K_Y+\Delta_Y)+G,$ where $G \geq 0$ and therefore we clearly have the following containment of stable base loci: $q^{-1}\mathbf{SB}(K_Y+\Delta_Y) \subset p^{-1}\mathbf{SB}(K_X+\Delta).$
		Similarly, note that for every sufficiently small ample $A$ on $X$,
		a $(K_X+\Delta)$-MMP step is a $(K_X+\Delta+A)$-MMP step. As $A$ is ample and $f$ birational, we can write $f_*A \sim_{\mathbb{Q}} H+E$, where $H$ is ample and $E$ effective. 
		Therefore  $q^{-1}\mathbf{SB}(K_Y+\Delta_Y+\frac{1}{n}H) \subset q^{-1}\mathbf{SB}(K_Y+\Delta_Y+ \frac{1}{n}f_*A)\subset  p^{-1}\mathbf{SB}(K_X+\Delta+ \frac{1}{n}A) $.
		As $\mathbf{B}_{-}(K_Y+\Delta_Y)=\bigcup_{n \geq 0} \mathbf{SB}(K_Y+\Delta_Y+\frac{1}{n}H)$ by \cite[Proposition 1.19]{asympt-baseloci} we conclude.
	\end{proof}

	Recall that, given a pair $(X,\Delta)$, a \emph{non-canonical centre} $V$ of $(X,\Delta)$ is the centre of a divisorial valuation $E$ with discrepancy $a(E, X, \Delta)<0$.  The following is a generalisation of \cite[Lemma 3.1]{HMX18} for arithmetic threefolds.
	
	\begin{proposition}\label{lemma:MMP_in_fam2}
        Let $X \to \Spec(R)$ be a projective contraction and suppose that $(X,B)$ is a $\bQ$-factorial klt threefold pair with $\mathbb{Q}$-boundary.
		Suppose the following conditions are satisfied:
		\begin{itemize}
		\item[(1)] $(X,X_k+B)$ is plt;
		\item[(2)] if $V$ is a non-canonical centre of $(X,X_k+B)$ contained in ${\mathbf{B}_{-}(K_{X}+B)}$, then $\dim (V_{k})=\dim (V) -1$.
		\end{itemize}
		Let $f \colon X\dashrightarrow Y$ be a step of a $(K_X+B)$-MMP over $R$. Then:
		\begin{enumerate}
			\item  if $f$ is a contraction of fibre type, then so is $f_k$;
			\item if $f$ is birational, then:
			\subitem(i) $f$ is a divisorial contraction;
			\subitem(ii) if $\Gamma:=f_\ast B$, then conditions (1) and (2) also hold for $(Y,Y_k+\Gamma)$.
		\end{enumerate} 
		In particular, if $f$ is a birational morphism then $h^0(X_t,m(K_{X_t}+B_t))=h^0(Y_t,m(K_{Y_t}+\Gamma_t))$ for all $t\in\Spec (R)$ and all $m\geq 0$ sufficiently divisible.
	\end{proposition}
	
	\begin{proof}
		If $f$ is a contraction of fibre type, hence $f_k$ is not birational by upper semi-continuity of the dimension of the fibres for proper morphisms (\cite[\href{https://stacks.math.columbia.edu/tag/0D4Q}{Tag 0D4Q}]{stacks-project}). 
		
		From now on, we assume that $f$ is birational. Suppose for contradiction that $f$ is a flip and consider the following diagram:
		
		\begin{equation*}
		\xymatrix{
			X \ar@{->}_{g}[rd] \ar@{-->}^{f}[rr]
			&
			& Y \ar@{->}^{g^+}[ld]\\
			&Z,}
		\end{equation*} 
		where $g$ is a $(K_X+B)$-flipping contraction.
		Note that $Y_k$ is irreducible since $f$ does not extract divisors, thus $f_k$ is birational. As $(X,X_k+B)$ is plt, so is $(Y,Y_k+\Gamma)$ hence both $X_k$ and $Y_k$ are normal by \autoref{invAdj3}. 
		
		We now derive the contradiction. Since $f$ is a flip, there exists a prime divisor $D$ on $Y_k$ such that its centre $P$ on $X_k$ is a closed point.
		Since $f_k$ is not an isomorphism at $N$ we have
		$a(D;X_k,B_k)<a(D;Y_k,\Gamma_k)\leq 0$	by \autoref{l:increase-discr}.
		Hence $P$ is a non-canonical centre of $(X_{k},B_k)$. Note that $P \subseteq \textup{Exc}(g) \subseteq \mathbf{B}_{-}(K_{X}+B)$ since $D$ is exceptional over $Z_k$. 
		Moreover $P$ is also a non-canonical centre of $(X,B+X_{k})$ as $$0 > \textup{totaldiscrep}(P, X_k, B_k) \geq \textup{discrep}(P, X, X_k+B), $$ by easy adjunction (\cite[Theorem 17.2]{FA}).
		So $P$ is an isolated non-canonical centre of $(X,X_k+B)$ contained in $\mathbf{B}_{-}(K_{X}+B)$, thus contradicting (2).
		
		Thus $f$, and therefore $f_k$, is a divisorial birational projective contraction. Condition (1) holds on $(Y,Y_k+\Gamma)$ immediately, so it remains to check condition (2).
		
		Suppose that $V$ is a non-canonical centre of $(Y,Y_{k}+\Gamma)$ and take a model $Z$ dominating $X$ and $Y$, and containing an exceptional divisor $E$ such that $V=\cent_Y(E)$ and $a(E, Y,Y_{k}+\Gamma) <0$. Then by \autoref{l:increase-discr} it must be that $a(E,X,X_k+B) \leq a(E,Y,Y_{k}+\Gamma) < 0$, hence the image, $W$, of $E$ on $X$ is a non-canonical centre of $(X,X_{k}+B)$. By \autoref{l-stable-base-loci} if $V \subseteq \mathbf{B}_{-}(K_{Y}+\Gamma)$ then we have $W \subseteq \mathbf{B}_{-}(K_{X}+B)$ as well. In which case $W$ is horizontal and hence so is $V$, therefore (2) holds as claimed.

		Since a $(K_X+B)$-MMP over $R$ is a $(K_X+X_k+B)$-MMP, we have that the map $(X_k,B_k) \rightarrow (Y_k, \Gamma_k)$ is a $(K_{X_k}+B_k)$-negative birational contraction and thus $h^0(X_t,m(K_{X_t}+B_t))=h^0(Y_t,m(K_{Y_t}+\Gamma_t))$ for all $t\in\Spec (R)$ and all $m\geq 0$ sufficiently divisible by \autoref{l-stein-invariance}.
	\end{proof}
	
\subsection{Rational polytopes of boundaries}
	
In this section we show how abundance for $\mathbb{R}$-pairs can be deduced from abundance for $\mathbb{Q}$-pairs, following a strategy due to Shokurov.
We recall the definition of the various polytopes we will need.
	
\begin{definition}
	Let  $X$ be a three-dimensional $\mathbb{Q}$-factorial excellent normal integral scheme and let $f \colon X \to T$ be a projective morphism over $R$ such that the image of $X$ in $T$ is positive-dimensional. 
	Fix a $\mathbb{Q}$-divisor $A\geq 0$. Let $V$ be a finite-dimensional, rational affine subspace of $\text{WDiv}_{\mathbb{R}}(X)$ containing no components of $A$.
	We have the following subsets of $\text{WDiv}_\mathbb{R}(X)$:
	\[V_{A}= \{A+B \colon B \in V\};\]		\[\mathcal{L}_{A}(V)=\{\Delta=A+B \in V_{A} \colon (X,\Delta) \textup{ is an lc pair}\};\]
	\[\mathcal{N}_{A}(V)=\{\Delta \in \mathcal{L}_{A}(V) \colon K_{X}+\Delta \textup{ is nef over } T\}.\]
\end{definition}
	
Recall that as long as there is a projective log resolution of $(X,A)$ together with $V$ the set $\mathcal{L}_{A}(V)$ is a rational polytope by the work of Shokurov (\cite{Sho92}), in particular this is true when $\dim (X) \leq 3$. Further if $(X,A+B)$ is klt and $(X,A+B')$ is lc then $(X,A+tB+(1-t)B')$ is klt for any $0 \leq t < 1$, so the set of klt pairs is open in $\mathcal{L}_{A}(V)$. In fact if $\mathcal{L}_{A}(V)$ contains a klt boundary, the entire interior consists of klt boundaries and the same is true for any sub-polytope.
	
The cone theorem, even the slightly weaker form proved in mixed characteristic in \cite{bhatt2020}, implies that $\mathcal{N}_{A}(V)$ is a rational polytope. 
	
\begin{lemma}[{\cite[Proposition 9.31]{bhatt2020}}]\label{neftope} 
	Fix a $\mathbb{Q}$-divisor $A \geq 0$ such that $(X,A)$ is a $\mathbb{Q}$-factorial klt threefold pair.
	Let $f\colon X \to T$ be a projective morphism of quasi-projective $R$-schemes.
	Then $\mathcal{N}_{A}(V)$ is a rational polytope.	
\end{lemma}
	
	We now infer abundance for pairs with $\mathbb{R}$-boundaries given the appropriate results for $\mathbb{Q}$-boundaries.
	
	\begin{proposition}\label{QtoR}
		Let $f \colon X \to T$ be a projective $R$-morphism of quasi-projective $R$-schemes such that $X$ is a $\mathbb{Q}$-factorial klt threefold and the image of $X$ in $T$ is positive-dimensional. 
		Suppose that for every $\mathbb{Q}$-divisor such that $(X,B)$ is klt and $K_{X}+B$ is $f$-nef, the divisor $K_{X}+B$ is $f$-semiample.
		Then $K_X+\Delta$ is $f$-semiample for any $\mathbb{R}$-divisor $\Delta$ such that $(X, \Delta)$ is klt and $K_{X}+\Delta$ is f-nef.
	\end{proposition}
	\begin{proof}
		We can write $\Delta= \sum_{i=1}^{n} t_{i}B_{i}$,  where $B_i$ are all distinct prime divisors on $X$ and $t_i \in \mathbb{R}_{\geq 0}$.
		Let $V$ be the $\mathbb{R}$-linear span of $B_i$ in $\text{WDiv}_\mathbb{R}(X)$. By \autoref{neftope} we have that $\mathcal{N}_{0}(V)$ is a rational polytope. Hence there are rational boundaries $D_{i} \in \mathcal{N}_{0}(V)$ such that $\Delta=\sum_i \lambda_{i} D_{i}$ where $\sum_i \lambda_{i} =1$. Since $(X,\Delta)$ is klt, by choosing $D_{i}$ sufficiently close to $\Delta$ we may suppose that each $(X,D_{i})$ is a klt pair with $\mathbb{Q}$-boundary and $K_X+D_i$ $f$-nef. 
		By assumption $K_{X}+D_{i}$ is $f$-semiample and thus so is $K_{X}+\Delta=\sum_i \lambda_{i} (K_{X}+D_{i})$.
	\end{proof}
	
	\section{Abundance for mixed characteristic threefolds}
	
	Given a log canonical pair $(X,\Delta)$ with a projective $R$-morphism $f \colon X \to T$ so that $K_{X}+\Delta$ is $f$-nef, then the abundance conjecture asserts that $K_{X}+\Delta$ is $f$-semiample. 
	In the case where $(X,\Delta)$ is a klt threefold pair and $K_{X}+\Delta$ (or even just $\Delta$) is big this is immediate by \autoref{t-bpf}. 
	We address the remaining cases in this section.
	
	The starting point of our proof is the abundance theorem for surfaces over excellent bases, which we now recall. 
	
	\begin{theorem}\label{abundance-dim2}
		Let $\pi \colon S \to T$ be a projective surjective morphism of normal quasi-projective schemes over $R$, 
		where  $(S,B)$ is a klt pair of dimension 2.  
		If $K_{S}+B$ is a $\pi$-nef $\mathbb{Q}$-Cartier $\mathbb{Q}$-divisor, then it is $\pi$-semiample.
	\end{theorem}
	
	\begin{proof}	
		If $T$ is a field then this is \cite[Theorem 1.2]{fujino2012log} for perfect fields and \cite{tanaka2020abundance} for imperfect fields. 
		Suppose from now on that $\dim (T) > 0$.
		If $K_{S}+B$ is big over $T$ then this follows immediately from the base-point-free theorem (\cite[Theorem 4.2]{Tan18}) with $D=2(K_{S}+B)$. Hence we may suppose that $\dim (T)=1$ and $K_{S}+B$ is not big. In this case we have $(K_{S}+B)|_{S_{K(T)}} \sim_{\mathbb{Q}} 0$ by the abundance theorem for curves (\cite[Lemma 9.22]{bhatt2020}) and the result follows by \autoref{lemma:EDsemiampleness}.
	\end{proof}
	
	The following is \cite[Lemma 2.17]{cascini2020relative}. We include the proof for completeness as the result is used often.

	\begin{lemma}\label{lemma:EDsemiampleness}
		Let $f\colon X \to Y$ be a contraction of integral, normal and excellent schemes. Suppose $L$ is an $f$-nef $\mathbb{Q}$-Cartier $\bQ$-divisor with $L|_{X_{K(Y)}} \sim_{\mathbb{Q}} 0$. If $Y$ is $\mathbb{Q}$-factorial and $f$ is equi-dimensional then $L \sim_{Y,\mathbb{Q}} 0$.
	\end{lemma}
	
	\begin{proof}
		Since $L|_{X_{K(Y)}} \sim_{\mathbb{Q}} 0$ we may write $L\sim_{Y, \mathbb{Q}} D\geq 0$ such that $D|_{X_{K(Y)}}=0$. If $C$ is any component of $D$ then $f(C)$ is a prime divisor, since $f$ is equi-dimensional. Thus, since $Y$ is $\mathbb{Q}$-factorial, it is enough to know that $L \sim_{\bQ,Y} 0$ after localisation about any codimension one point of $Y$. In particular we may suppose that $Y= \Spec (R)$ for some DVR $R$ with closed point $P$.
		
		Let $\left\{G_i \right\}_{i=1}^n$ be the irreducible components of the special fibre $F=f^*P$, so that by construction $D = \sum_{i=1}^n a_i G_i$ for certain $a_i \geq 0$.	
		
		We introduce $r:= \min \left\{ t \mid D -tF \leq 0 \right\}$. We are left to show that $D-rF=0$. 
		If not, up to rearranging the order of $G_i$, we have $D-rF=-\sum_{i=2}^n l_i G_i \equiv_Y 0$, with $l_2 >0, l_i \geq 0$ and $G_{1}$ meeting $G_{2}$. Note that $(rF-D)$ is effective curve not containing $G_{1}$ but intersecting it. Hence there must be a curve $C$ on $G_{1}$ with $(rF-D) \cdot C >0$, but $rF-D \sim_{T} -D$ and $D$ is nef, a contradiction. Therefore $D-rF=0$ as claimed.
	\end{proof}

	The following gives a sufficient condition for a nef divisor to be EWM together with a a well-behaved birational modification (cf.~ \cite[Lemma 9.25]{bhatt2020}).
	
\begin{lemma}\label{two}
	Let $X \to T$ be a projective contraction of normal, integral, quasi-projective $R$-schemes.
	Let $L$ be a $\mathbb{Q}$-Cartier $\mathbb{Q}$-divisor on $X$, nef over $T$ such that $L|_{X_{K(T)}}$ and $L|_{X_{\mathbb{Q}}}$ are semiample.
	Assume $\dim (X) \leq 3$ and $L$ is not big. 
	Then $L$ is EWM and there is a commutative diagram of proper algebraic spaces over $T$:
	\[
	\begin{tikzcd}
	W \arrow[d, "g"] \arrow[r, "\phi"] & X \arrow[d, "f"] \\
	Y \arrow[r, "\pi"]                 & Z      ,        
	\end{tikzcd}
	\]
	such that 
    \begin{enumerate}
	    \item $f$ is the EWM contraction associated to $L$;
		\item $\phi$ and $\pi$ are proper birational contractions;
		\item  $g$ is equi-dimensional, $W$ is a $T$-projective normal scheme, and $Y$ is a projective regular scheme over $T$ of dimension $\leq 2$;
			\item $g$ agrees with the map induced by $\phi^*L$ over the generic point of $Z$;
			\item  there exists a $\mathbb{Q}$-Cartier $\mathbb{Q}$-divisor $D$ on $Y$ such that $\phi^{*}L \sim_{\mathbb{Q}} g^{*}D$.
	\end{enumerate}  
\end{lemma}
	
\begin{proof}
    Note that if $\dim(T)=0$ there is nothing to prove, hence we can assume $\dim(T)\geq 1$. By \cite[Lemma 9.24]{bhatt2020} and its proof we can find a diagram of schemes over $T$:
    \[
	\begin{tikzcd}
	W \arrow[d, "g"] \arrow[r, "\phi"] & X  \\
	Y           &      ,
	\end{tikzcd}
	\]
	such that $\phi$ is birational and there exists a $\mathbb{Q}$-Cartier $\mathbb{Q}$-divisor $D$ on $Y$ such that (c)-(e) hold.
    By \autoref{as-ext} and \autoref{pp-EWM} it is sufficient to show that $D$ is EWM to conclude. If $\dim(Y) \leq 1,$ the result is trivial and if $\dim(Y)=2$, we apply \cite[Lemma 2.48]{bhatt2020}.
\end{proof}	

If $f$ is equi-dimensional, it is possible to prove a suitable semiampleness result.
\begin{proposition}\label{EDsemiampleness2}
	Let $X \to T$ be a projective contraction of normal quasi-projective schemes over $R$, where $\dim(X) \leq 3$.
	Let $L$ be an EWM $\mathbb{Q}$-Cartier $\mathbb{Q}$-divisor on $X$ such that its associated EWM contraction $f \colon X \to Z$ is equi-dimensional.
	If $L|_{X_{K(T)}}$ and $L|_{X_{\mathbb{Q}}}$ are semiample, then $L$ is semiample.
	\end{proposition}

\begin{proof}
Without loss of generality we may assume $\dim(T)\geq 1$. If $L$ is big and $f$ is equi-dimensional, then $L$ is necessarily ample and we conclude.
We can thus suppose $L$ is not big. We can then apply \autoref{two} and thus there exists a commutative diagram of proper algebraic spaces over $T$
\[\begin{tikzcd}
		W \arrow[d, "g"] \arrow[r, "\phi"] & X \arrow[d, "f"] \\
		Y \arrow[r, "\pi"]             & Z  ,                
		\end{tikzcd}\]
		such that the following hold:
		
		\begin{enumerate}
			\item $W,X,Y$ are normal $T$-projective schemes and $\dim (Y) \leq 2$;
			\item the vertical maps $f$ and $g$ are equi-dimensional;
			\item the horizontal maps $\phi$ and $\pi$ are proper and birational;
			\item there exists a $\mathbb{Q}$-Cartier $\bQ$-divisor $D$ on $Y$ such that $\phi^{*}L \sim_{\mathbb{Q}} g^{*}D$.
		\end{enumerate}
		Since $Z$ is normal, there is an open immersion of a regular scheme $U \to Z$ containing every codimension 1 point of $Z$ by \autoref{l-cod1}.
		By (b), $X_{U} \to U$ satisfies the assumptions of \autoref{lemma:EDsemiampleness} and thus $L|_{X_{U}} \sim_{U,\mathbb{Q}} 0$. 
		
		If $\dim(Z)=1$, then we conclude immediately, so we can suppose $Z$ is a surface. Therefore $Z \setminus U$ consists of finitely many points.  
		Then we may choose $S$ to be a general hyperplane on $X$ such that $S$ meets each fibre over $Z \setminus U$ at only finitely many points. 
		
		Note that $L|_{S}$ is clearly big and moreover if $C$ is any curve on $S$ with $L \cdot C =0$, then $C$ must be contracted by $X \to Z$ as $f$ is the EWM contraction associated to $L$. In particular $C$ is contained in some fibre of $f$ and by construction $C$ is not contained in a fibre over $Z \setminus U$, as $S$ contains no such curves. Thus in fact $C \subseteq X_{U}$. Therefore $\mathbb{E}(L|_{S}) \subseteq X_{U}$ and so $L|_{\mathbb{E}(L|_{S})}$ is semiample since $L|_{X_{U}}$ is. As $L|_{S_{\mathbb{Q}}}$ is semiample by assumption, we conclude that $L|_{S}$ is semiample by \cite[Theorem 6.1]{witaszek2020keels}. 
		
		Let $S'$ be the strict transform of the surface $S$ on $W$, which must dominate $Y$. Let $\phi_{S'}, g_{S'}$ be the restrictions of $\phi, g$ to $S'$. Then $(\phi^{*}L)|_{S'} \sim_{\mathbb{Q}}\phi_{S'}^{*}(L|_{S})\sim_{\mathbb{Q}}g_{S'}^{*}D$ and since $L|_{S}$ is semiample and $Y$ is normal, we must have that $D$ is semiample by \autoref{pullback}. In turn this implies that $L$ is semiample as $\phi^*L \sim_{\mathbb{Q}} g^*D$.	
\end{proof}	
	
The following is a useful MMP technique to reduce to the case of equi-dimensional morphisms.
%

\begin{proposition}\label{three}
    Suppose that none of the residue fields of $R$ have characteristic $2,3$ and $5$. 
	Let $X \to T$ be a projective $R$-morphism of quasi-projective $R$-schemes with positive-dimensional image and let $(X,B)$ be a $\mathbb{Q}$-factorial klt threefold pair.
	Suppose that
	\begin{enumerate}
	    \item $K_X+B$ is an EWM $\mathbb{Q}$-divisor over $T$ with  $h \colon X \to Z$ being the associated EWM contraction;
		\item $Z$ has dimension 2.
	\end{enumerate}
	Then there exists a $(K_X+B)$-trivial birational contraction $(X,B) \dashrightarrow (X', B')$ over $Z$ such that $X' \to Z$ is equi-dimensional. 
\end{proposition}

\begin{proof}
	Let $z \in Z$ be a closed point such that the fibre $h^{-1}(z)$ is not one-dimensional. By upper semi-continuity of fibre dimensions for proper morphisms (\cite[\href{https://stacks.math.columbia.edu/tag/0D4Q}{Tag 0D4Q}]{stacks-project}) $h^{-1}(z)$ must contain an irreducible surface $F$.
		
		Take $t >0$ with $(X,B+tF)$ klt and run a $(K_X+B+tF)$-MMP over $T$. 
		We now show that this is an MMP over $Z$ as well.	
		Let $C$ be a curve generating an extremal $(K_X+B+tF)$-negative ray. As $(K_{X}+B)$ is nef over $T$, then $F\cdot C <0$. Therefore $C \subseteq F$ and since $F$ is contracted by $h$ to a point, so too is $C$. 
		 By definition $X \to Z$ contracts only $(K_X+B)$-trivial curves.
        From this we can conclude that the $(K_X+B+tF)$-MMP over $T$ is also a $(K_X+B+tF)$-MMP over $Z$ by \autoref{as-ext}. Moreover, the steps of this MMP are $(K_X+B)$-trivial.
		

		Since this is an MMP of a pseudo-effective klt pair over $T$ it terminates by \autoref{MMP}. In fact we claim it terminates when the strict transform of $F$ is contracted.
		
		If $X\dashrightarrow X'$ does not contract $F$ then its transform on $X'$ remains the divisorial part of a fibre, so to establish this claim it is sufficient to show that such divisorial part is never nef. 
		By abundance on the generic fibre $(X'_{K(Z)},B'_{K(Z)})$ (\autoref{abundance-dim2}) we can apply \autoref{two} to find a commutative diagram 
		\[
		\begin{tikzcd}
        W \arrow[d, "g"] \arrow[r, "\phi"] & X' \arrow[d, "f'"]  \\
		Y\arrow[r, "\pi"]           & Z,              
		\end{tikzcd}
		\]
		where $g$ is equi-dimensional, $Y$ is a regular projective surface over $T$ and $\phi, \pi$ are proper birational.
		Let $F'$ be the strict transform of $F$ on $W$. Then $g(F')=\gamma$ must be an irreducible curve by equi-dimensionality and $\pi_*(\gamma)=z$. Choose a general curve $C$ in $F'$ such that $g(C)=\gamma$.
		Since is $D$ big, we write $D\sim_{\mathbb{R}} A+E$, for $A$ ample and $E$ effective by Kodaira's lemma.
		By Bertini theorems (\cite[Theorem 2.15]{bhatt2020}) we can choose a general $H \sim_{\mathbb{R}} A$ meeting $\gamma$ transversally. 
		Then $g^{*}H \cap C$ is a finite set of points and $g^{*}H$ is not contracted by $\phi$ as $H$ is general.
		We have $K_{X}+B \sim_{\mathbb{R}}\phi_{*}g^{*}(A+E) \sim_{\mathbb{R}}\phi_{*}g^{*}H+S$ where $S \geq 0$. As $C$ is general in $F'$ we have
		$$\phi_{*}g^{*}H \cdot \phi_{*}C >0 \textup{ and } (K_{X}+B)\cdot \phi_{*}C=g^{*}D \cdot C=D \cdot \gamma=0 \text{ as } \pi_*\gamma=z,$$ so we have $S \cdot C <0$. Since $C$ is general in $F'$ we must have that $F$ is contained in the support of $S$ and $F \cdot \phi_{*}C <0$.
		
		Since there are only finitely many closed points $z \in Z$ for which the fibres are not one-dimensional, we can repeat the above process a finite number of times and we terminate with a crepant model $(X',B')$ which is equi-dimensional over $Z$.	
	\end{proof}

The following result states that it's sufficient to prove abundance for a single minimal model inside it birational class.

	
	\begin{lemma}\label{l-inv-model}
	Suppose that $(X,B)$ is a pseudo-effective log canonical pair, projective over $T$. 
	Let $\phi_{i} \colon (X, B) \dashrightarrow (Y_{i},B_{i})$ be two weak log canonical models for $(X,B)$ over $T$. 
	Then $K_{Y_{1}}+B_{1}$ is semiample over $T$ if and only if $K_{Y_{2}}+B_{2}$ is so.
\end{lemma}
	
	\begin{proof}
		Let $f \colon Z \to X$ be a projective birational contraction of normal schemes together with proper birational contractions $g_i \colon Z \to Y_i$. 
		As $Y_i$ are weak log canonical models, we can write $$K_Z+\Delta_Z = g_i^*(K_{Y_i}+B_i) + E_i,$$
		where $\Delta_Z$ defined by crepant pullback $K_Z+\Delta_Z=f^*(K_X+B)$, and $E_i$ are effective and $g_i$-exceptional divisors.
		Consider 
		$$g_1^*(K_{Y_1}+B_1)-g_2^*(K_{Y_2}+B_2)=E_2-E_1. $$
		In particular, $E_2-E_1$ is $g_2$-nef and therefore by the negativity lemma (\cite[Lemma 2.14]{bhatt2020}) we conclude that $E_2 -E_1 \leq 0$. By symmetry, we conclude that $E_2=E_1$. Therefore $g_1^*(K_{Y_1}+B_1)=g_2^*(K_{Y_2}+B_2)$. In particular, $K_{Y_1}+B_1$ is semiample over $T$ iff $K_{Y_2}+B_2$ is so.
	\end{proof}
	
We are now ready to prove the abundance theorem for klt threefolds over a positive-dimensional base which is not of pure characteristic $0$.
	
	\begin{theorem}\label{abundance}
	    Suppose that none of the residue fields of $R$ have characteristic $2,3$ and $5$. 
	    Let $\pi \colon X \to T$ be a projective morphism of quasi-projective $R$-schemes such that $\pi(X)$ is positive-dimensional and it contains a closed point of characteristic $p>5$.
		Suppose $(X,B)$ is a klt threefold pair with $\mathbb{R}$-boundary. If $K_X+B$ is $\pi$-nef, then it is $\pi$-semiample.
	\end{theorem}
	
	\begin{proof}
		By Stein factorisation we can assume $\pi$ to be a contraction of normal schemes, so $\dim(T) \geq 1$.
		If $X$ is not $\mathbb{Q}$-factorial, then we may freely replace it with a $\mathbb{Q}$-factorialisation by \autoref{Q-factorial} and \autoref{l-inv-model}. Moreover by \autoref{QtoR} we can suppose that $\Delta$ is a $\mathbb{Q}$-boundary (we reduce to this case where we can apply the results of \cite{witaszek2020keels} which apply only to $\mathbb{Q}$-Cartier $\bQ$-divisors).
		As the dimension of $X_{K(T)}$ is at most 2, we conclude by \cite[Lemma 9.22]{bhatt2020} and \autoref{abundance-dim2} that $\kappa(K_{X_{K(T)}}+B_{K(T)}) \geq 0$. We now divide the proof according to the value of $\kappa(K_{X_{K(T)}}+B_{K(T)})+\dim (T)$.

		\begin{case} 
		$\kappa(K_{X_{K(T)}}+B_{K(T)})+\dim (T)=3$.
		\end{case}
		In this case, $K_X+B$ is big and we can conclude by applying \autoref{t-bpf} to $L:=2(K_{X}+B)$.
		
		\begin{case}
		$\kappa(K_{X_{K(T)}}+B_{K(T)})+\dim (T)=2$.
		\end{case}
		By \autoref{abundance-dim2},  $K_{X_{K(T)}}+B_{K(T)}$ and $K_{X_\mathbb{Q}}+B_{\mathbb{Q}}$ are semiample $\mathbb{Q}$-divisors. As $K_X+B$ is not big, by \autoref{two} then $K_X+B$ is EWM and we denote by $f \colon X \to Z$ the associated EWM contraction.
		By \autoref{three} and \autoref{l-inv-model} we may replace $X$ so that $X \to Z$ is equi-dimensional. We then apply \autoref{EDsemiampleness2} to deduce that $K_{X}+\Delta$ is $\pi$-semiample.
        
        \begin{case}
		$\kappa(K_{X_{K(T)}}+B_{K(T)})+\dim (T)=1$.
		\end{case} 
		The hypothesis $\dim (T) \geq 1$ implies  $\kappa(K_{X_{K(T)}}+B_{K(T)})=0$. Then $\pi \colon X \to T$ is flat, since $T$ is a Dedekind scheme and $X$ is integral by \cite[Proposition 9.7]{Ha77}. Since $K_{X_{{K(T)}}}+ B_{{K(T)}}$ is semiample by \autoref{abundance-dim2}, we conclude $K_X+B$ is semiample by \autoref{lemma:EDsemiampleness}. 
\end{proof}

\begin{remark}
	While in this section we worked on threefolds over rings whose residue fields have characteristic $p \neq 2, 3$ and $5$, this is just due to the current state of the art on the MMP. 
	The arguments in the section for $\kappa(K_{X_{K(T)}}+B_{K(T)})+\dim(T) \leq 2$ work as long as the MMP results are known to hold.
	In particular, abundance holds for mixed characteristic threefolds over a Dedekind domain with residue characteristics different from $2$ and $3$ by \cite{XX22}.
\end{remark}
	
\section{Applications to invariance of plurigenera} \label{s-inv-plurigenera}

	In this section, $R$ will always be an excellent DVR with residue field of characteristic $p>5$. 
	The purpose of this section is to generalise the asymptotic invariance of plurigenera proven in \cite[Theorem 3.1]{EH} to families of \emph{non-log-smooth} surface pairs, as well as excellent DVRs with non-perfect residue field. Similar results in characteristic 0 are proven in \cite{HMX13, HMX18}.
	The first case we discuss is the asymptotic invariance for families of good minimal models.
	
\begin{theorem}\label{thm:ADIOP_SA}
	Let $\pi \colon X \to\Spec (R)$ be a  projective contraction and suppose $(X,B)$ is a three-dimensional klt pair with $\mathbb{Q}$-boundary.
	Assume that $(X,X_k+B)$ is plt and $K_X+B$ is semiample over $R$.
	Suppose one of the following holds:
	\begin{enumerate}
		\item $\kappa(K_{X_k}+B_k)\neq 1$; or
		\item  $B_k$ is big over $\Proj R(K_{X_k}+B_k)$.
	\end{enumerate} 
	Then there exists an $m_{0} \in \mathbb{N}$ such that 
	$$h^0(X_K,m(K_{X_K}+B_K))=h^0(X_k,m(K_{X_k}+B_k))$$
	for all $m\in m_0\bN$.
\end{theorem}
	
	We start by showing the normality of the central fibre.
	
	\begin{proposition}\label{p-gen-case}
	Let $\pi \colon X \to\Spec (R)$ be a  projective contraction and suppose $(X,B)$ is a three-dimensional klt pair with $\mathbb{Q}$-boundary.
	Suppose that $(X,X_{k}+B)$ is plt. If $f \colon X \to Z$ is a projective birational morphism of normal schemes over $R$ such that $-(K_{X}+B)$ is $f$-nef, then $X_k$ and $Z_k$ are normal and $f_{k,*}\ox[X_{k}]=\ox[Z_{k}]$.
	\end{proposition}
	
	\begin{proof}
		By \autoref{invAdj3} the central fibre $X_{k}$ is normal. 
		As $f$ is surjective, generically finite, and $X$ and $Z$ are equidimensional over $R$, then $f_{k}$ is generically finite as well. In particular $-(K_{X_{k}}+B_{k})$ is $f_{k}$-big and $f_k$-nef. We conclude by \autoref{invAdj2}.
	\end{proof}
	
	The next result is used to reduce the non-$\mathbb{Q}$-factorial case to the $\mathbb{Q}$-factorial one.
	
	\begin{lemma}\label{l-reduce-Q-fac}
		Let $Y \to \Spec(R)$ be a projective contraction such that $(Y,Y_{k}+\Delta)$ is a plt threefold.
		Let $f \colon Y \to X$ be a $(K_{Y}+\Delta)$-trivial small birational contraction over $R$ with $B=\pi_{*}\Delta$. 
		Then
		$$h^{0}(X_{k},m(K_{X_k}+B_{k}))= h^{0}(Y_{k},m(K_{Y_{k}}+\Delta_{k}))$$ 
		for all $m$ sufficiently large and divisible. 
	\end{lemma}

	\begin{proof}
		
		As $f$ is small, the central fibre $Y_k$ is irreducible. 
		By \autoref{t-bpf},	$K_Y+Y_k+\Delta \sim_{\mathbb{Q}} f^*(K_X+X_k+B)$.
		Then by \autoref{p-gen-case}, $Y_k$ and $X_k$ are both normal.
		As $f$ is the semiample contraction associated to $K_Y+\Delta$ over $X$ and it is birational, we conclude by \autoref{l-stein-invariance}.
	\end{proof}
	
	We now discuss the delicate case of invariance for plurigenera where the Kodaira dimension is 1 and the boundary is big over the ample model.

	\begin{proposition}\label{p-1-case}
		
		Suppose $(X,B)$ is a $\mathbb{Q}$-factorial three-dimensional klt pair with $\mathbb{Q}$-boundary. Let $\pi \colon X \to \Spec(R)$ be a projective contraction such that $(X,X_{k}+B)$ is plt. Suppose further that
		\begin{enumerate}
			\item $K_X+B$ is semiample and $f\colon X \to Z$ is its semiample contraction over $R$;
			\item $\kappa(K_{X_{k}}+B_{k})=1$;
			\item $B_{k}$ is big over $Z_{k}$.
		\end{enumerate}
		Then there exists an $m_{0} \in \mathbb{N}$  such that 
		$$h^0(X_K,m(K_{X_K}+B_K))=h^0(X_k,m(K_{X_k}+B_k))$$
		for all $m\in m_0\bN$.	
	\end{proposition}
	
	\begin{proof}
		
		As $K_X$ is not pseudo-effective over $Z$, a $K_{X}$-MMP over $Z$ with scaling of $A$ terminates by \autoref{MMP} with a Mori fibre space $X \to Z'$ over $Z$ as $B_k$ is big over $Z_k$. 
		Since each step is $(K_{X}+B)$-trivial and $X_{k}$ is not contracted, $(X,X_{k}+B)$ remains plt during the MMP and so $X_{k}$ stays irreducible and normal by \autoref{invAdj3}.
		
		Consider a step of this MMP $\phi \colon X \dashrightarrow X'$ and let $B':=\phi_*B$.  We have
		
		\[\begin{tikzcd}
			X \arrow[swap, rd, "g"] \arrow[rr, "\phi", dashed] &                 & X' \arrow[ld, "h"] \\
			& Y \arrow[d] &                        \\
			& Z ,         &                       
		\end{tikzcd}\]
        where $h$ may either be an isomorphism or a small projective birational contraction. As $g$ is $(K_X+B)$-trivial, by \autoref{p-gen-case} we have that $Y_k$ is irreducible and normal as well. Let now $(Y,\Xi)$ be the induced pair on $Y$: we then have
		\[h^{0}(X_{k},m(K_{X_k}+B_{k}))=h^{0}(Y_{k},m(K_{Y_{k}}+\Xi_k))=h^{0}(X'_{k},m(K_{X'_k}+B'_{k}))\] 
		where the first equality also follows from \autoref{p-gen-case} and \autoref{l-stein-invariance} and the latter is \autoref{l-reduce-Q-fac}. Clearly also the sections of $m(K_{X}+B)$ are preserved by this MMP for large divisible $m > 0$.
		
		Hence we can now suppose $X$ admits a Mori fibre space $X \to Z'/Z$, with $B_{k}$ big over $Z_{k}$.
		
		We claim that $Z'=Z$.
		Indeed, suppose for contradiction that there exists a divisor $D$ on $Z'$ which is contracted by $Z' \to Z$. Since $\dim (Z)=2$, $D$ must be contained in $Z'_k$. But then $f^{-1}D$ is a surface inside $X_{k}$, which is irreducible by assumption. It cannot be that $X_{k}$ is contracted to a point over $Z$, thus no such $D$ exists and we have $Z=Z'$.
		
		In particular $-K_{X}$ is ample over $Z$ and hence by \autoref{invAdj2} we have that $f_{k,*} \ox[X_{k}]=\ox[Z_{k}]$ and the result follows from \autoref{l-stein-invariance}.
	\end{proof}

	We can now prove the asymptotic invariance of plurigenera for a family of minimal models of klt pairs.
	
\begin{proof}[Proof of \autoref{thm:ADIOP_SA}]
	By \autoref{l-reduce-Q-fac} we can suppose $X$ is $\mathbb{Q}$-factorial.
	We have $\kappa(K_{X_k}+B_k)=\kappa(K_{X_K}+B_K)=\kappa$, since the Iitaka dimension is deformation invariant for semiample line bundles by \autoref{lemma_DIOK}. Let $f \colon X\to Z$ be the semiample contraction over $R$, so that $K_X+B\sim_{\bQ}f^\ast A$ for some ample $\bQ$-divisor. 
	We now divide in various cases.
		
	If $\kappa=0$, then $K_X+B\sim_{\bQ}0$ and hence we conclude by \autoref{lemma_DIOK}.
	If $\kappa=1$, then this is \autoref{p-1-case}. 
	Finally in the case $\kappa=2$ we conclude by \autoref{p-gen-case} and \autoref{l-stein-invariance}. 	
\end{proof}

	Putting these results together with \autoref{lemma:MMP_in_fam2} and the abundance theorem \autoref{abundance} we deduce an asymptotic invariance result for plurigenera on suitable families. To clarify the conditions we need to impose on the non-canonical locus of the family, we recall an example due to Kawamata.
	
	\begin{example}	\label{ex-kawamata}
		Let $R$ be an excellent DVR and consider the following diagram of $R$-flat families:
		
		\[\begin{tikzcd}
			X \ar[rr, dashed, "\phi"] \arrow[rd, "g",swap] \ar[rdd, bend right,
			"f",swap] &   & X^{+}  \arrow[ld, "g^+"] \ar[ldd, bend left,
			"f^{+}" ',swap]  \\
			&  Z    \ar[d]                             & 	\\
			& \Spec(R) &
		\end{tikzcd}
		\] 
		where
		\begin{enumerate}
			\item $X$ is a terminal threefold and the central fibre $X_k$ has an isolated klt non-canonical singular point;
			\item $g$ is an extremal $K_{X}$-negative flipping contraction;
			\item $X^{+}$ is regular.
		\end{enumerate}
		A local model is given by the Francia flip in \cite[Example 4.3]{Kaw99}. 
		As explained by Kawamata, the homomorphism
		$f_*\mathcal{O}_{X}(mK_{X}) \to f_* \mathcal{O}_{X_k}(mK_{X_k})$ is not surjective.
	
		Note that this situation is excluded by condition (2) of \autoref{lemma:MMP_in_fam2} and \autoref{thm:ADIOP_final2}.
		Indeed $\textbf{B}_{-}(K_{X_k})$ clearly contains the flipped locus of $g$, which must contain the non-canonical singular points $p$ of $X_k$. As $X$ is terminal, $p$ is not the restriction of a horizontal non-canonical centre of $X$.
	\end{example}

Imposing the conditions of \autoref{lemma:MMP_in_fam2}, we are able to prove the invariance of plurigenera for families of klt surfaces from \autoref{thm:ADIOP_SA}.
	
	\begin{theorem}\label{thm:ADIOP_final2}
		Let $X \to \Spec(R)$ be a projective contraction and suppose that $(X,B)$ is a threefold klt pair with $\mathbb{Q}$-boundary.
		Suppose that all of the following are satisfied:
		
		\begin{enumerate}
			\item[(1)] $(X, X_{k}+B)$ is plt;
			\item[(2)] if $V$ is a non-canonical centre of $(X,X_k+B)$ contained in ${\mathbf{B}_{-}(K_{X}+B)}$, then $\dim (V_{k})=\dim (V) -1$.
		\end{enumerate}
		Then $X_k$ is normal. Suppose further that at least one of the following holds:
		\begin{enumerate}
			\item $\kappa(K_{X_{k}}+B_{k}) \neq 1$; or
			\item $B_{k}$ is big over $\textup{Proj}(K_{X_{k}}+B_{k})$
		\end{enumerate}	
		Then there is $m_{0} \in \mathbb{N}$ such that 
		$$h^{0}(X_{K},m(K_{X_{K}}+B_{K}))=h^{0}(X_{k},m(K_{X_{k}}+B_{k}))$$
		for all $m \in m_{0}\mathbb{N}$.
		
	\end{theorem}
	
	\begin{proof}
		By \autoref{l-reduce-Q-fac} we can suppose $X$ is $\mathbb{Q}$-factorial.
		We may run a $(K_X+B)$-MMP over $R$ which terminates by \autoref{MMP}. 
		We call $(Y,\Gamma)$ the end-product of this MMP. Since $(X,B)$ satisfies conditions (1)-(2) of \autoref{lemma:MMP_in_fam2} we deduce $h^{0}(X_{k},m(K_{X_{k}} + B_{k}))=h^0(Y_k, m(K_{Y_k}+\Gamma_k))$ for all sufficiently divisible $m$. 
		In the case where $\kappa(K_{X_{k}}+B_{k})=1$, the condition that $B_{k}$ is big over $\textup{Proj}(K_{X_{k}}+B_{k})$ is also preserved by the MMP.
		
		If $K_{X}+B$ is pseudo-effective then $K_Y+\Gamma$ is nef over $R$. Therefore, by \autoref{abundance}, $K_{X}+B$ is semiample and the result then follows from \autoref{thm:ADIOP_SA}. 
		If $K_X+B$ is not pseudo-effective over $R$, then there is a Mori fibre space structure $(Y,\Gamma) \to Z$. This ensures that neither $K_{Y}+\Gamma$ nor $K_{Y_{k}}+\Gamma_{k}$ are pseudo-effective and thus the result holds trivially. 
	\end{proof}

	\begin{remark}	
		The $p>5$ assumption is essential to the adjunction type results used  in \autoref{p-gen-case}. Even if the MMP was known in lower characteristic, our arguments in this section would not extend immediately.
	\end{remark}
	
	\bibliographystyle{amsalpha}
	\bibliography{refs}
\end{document}